\documentclass[a4paper,12pt]{article}     
\usepackage{amsmath,amscd,amssymb}        
\usepackage{latexsym}                     
\usepackage[english, german]{babel}                

\usepackage{theorem}                 

\newtheorem{theorem}{Theorem}

\newenvironment{definition}
{\smallskip\noindent{\bf Definition\/}:}{\smallskip\par}

\newenvironment{remark}
{\smallskip\noindent{\bf Remark\/}.}{\smallskip\par}

\newenvironment{proof}{\begin{ProofwCaption}{Proof}}{\end{ProofwCaption}}
\newenvironment{proof*}[1]{\begin{ProofwCaption}{{#1}}}{\end{ProofwCaption}}
\newenvironment{ProofwCaption}[1]%
  {\addvspace\theorempreskipamount \noindent{\it #1.}\rm}%
  {\qed \par \addvspace\theorempostskipamount}
\newcommand{\qedsymbol}{\mbox{$\Box$}}
\newcommand{\qed}{\hfill\qedsymbol}

\newcommand{\CC}{{\mathbb C}}

\newcommand{\RR}{{\mathbb R}}
\newcommand{\ZZ}{{\mathbb Z}}

\title{Orbifold Euler characteristics for dual invertible polynomials}
\author{Wolfgang Ebeling and Sabir M.~Gusein-Zade
\thanks{Partially supported by the DFG Mercator program (INST 187/490-1), the Russian government grant 11.G34.31.0005, RFBR--10-01-00678,
NSh--8462.2010.1.
Keywords: invertible polynomials, group actions, orbifold Euler characteristic.
AMS 2010 Math. Subject Classification: 14J33, 32S55, 57R18.
}
}
\date{}

\begin{document}
\selectlanguage{english}

\maketitle

\begin{abstract}
To construct mirror symmetric Landau--Ginzburg models, P.~Berg\-lund, T.~H\"ubsch and M.~Henningson considered a pair $(f,G)$ consisting of an invertible polynomial $f$ and an abelian group $G$ of its symmetries together with a dual pair $(\widetilde{f}, \widetilde{G})$. Here we study the reduced orbifold Euler characteristics of the Milnor fibres of $f$ and $\widetilde f$ with the actions of the groups $G$ and $\widetilde G$ respectively and show that they coincide up to a sign.
\end{abstract}

\section*{Introduction}  
P.~Berglund and T.~H\"ubsch \cite{BH1} proposed a method to construct some 
mirror symmetric pairs of manifolds. Their construction involves a polynomial $f$ of a special form, a so called (non-degenerate) {\em invertible} one, and its {\em Berglund--H\"ubsch transpose} $\widetilde f$. In their paper, these polynomials appeared as potentials of Landau--Ginzburg models. The construction of \cite{BH1} was generalized in \cite{BH2} to orbifold Landau--Ginzburg models described by pairs $(f, G)$, where $f$ is a (non-degenerate) invertible polynomial and $G$ is a (finite) abelian group of symmetries of $f$. For a pair $(f, G)$ one defines the dual pair $(\widetilde{f}, \widetilde{G})$. In \cite{BH2, KY}, there were described some symmetries between invariants of the pairs $(f, G)$ and $(\widetilde{f}, \widetilde{G})$ corresponding to the orbifolds defined by the equations $f=0$ and $\widetilde{f}=0$ in weighted projective spaces.  Moreover, in \cite{EG, ET2, Taka}, it was observed that the singularities defined by $f$ and $\widetilde f$ have some duality properties.

One can consider the Milnor fibre $V_f$ of an invertible polynomial $f$ with the action of a group of symmetries of $f$. The monodromy transformation of the Milnor fibre is also induced by a symmetry transformation of $f$. 

In \cite{EG}, there was defined an equivariant version of the monodromy zeta function and an equivariant version of the, so called, Saito duality. The latter one is an analogue of the Fourier transform from the Burnside ring of a finite abelian group $G$ to the Burnside rings of its group of characters $G^*$. It was shown that the equivariant monodromy zeta functions of Berglund--H\"ubsch dual polynomials $f$ and $\widetilde f$ with respect to their maximal abelian symmetry groups $G_f$ and $G_{\widetilde f}$ are, up to a sign, Saito dual to each other. This result and its proof inspired the considerations of this paper.

In \cite{DHVW1,DHVW2}, there was introduced the concept of the orbifold Euler characteristic. Here we consider the notion of a {\em reduced} orbifold Euler characteristic. For a pair $(f, G)$ as above, we consider the reduced orbifold Euler characteristic $\overline{\chi}(V_f, G)$ of the Milnor fibre $V_f$ of $f$ with the $G$-action on it. For a non-degenerate $f$ (i.e.\ for $f$ with an isolated critical point at the origin), it can be regarded as an orbifold Milnor number of the pair $(f, G)$.
We show that the reduced orbifold Euler characteristics $\overline{\chi}(V_f, G)$ and $\overline{\chi}(V_{\widetilde f}, \widetilde{G})$ are equal up to a sign. 
This gives additional symmetries between the dual pairs  $(f, G)$ and $(\widetilde{f}, \widetilde{G})$.

\section{Invertible polynomials}
A quasihomogeneous polynomial $f$ in $n$ variables is called 
{\em invertible} (see \cite{Kreuzer}) if the number of monomials in it coincides with the number of variables $n$, i.e., if it is of the form 
\begin{equation}\label{inv} 
f(x_1, \ldots, x_n)=\sum\limits_{i=1}^n a_i \prod\limits_{j=1}^n x_j^{E_{ij}} 
\end{equation} 
for some coefficients $a_i\in\CC^\ast$ and for a matrix 
$E=(E_{ij})$ with non-negative integer entries and with $\det E\ne 0$. 
Without loss of generality one may assume that $a_i=1$ 
for $i=1, \ldots, n$ (this can be achieved by a rescaling of the variables $x_j$) and that $\det E>0$. 
An invertible quasihomogeneous polynomial $f$ is {\em non-degenerate} if it has (at most) an isolated critical point at the origin in $\CC^n$. 

The {\em Berglund-H\"ubsch transpose} $\widetilde{f}$ of the invertible polynomial (\ref{inv}) is defined by 
$$ 
\widetilde{f}(x_1, \ldots, x_n)=\sum\limits_{i=1}^n a_i \prod\limits_{j=1}^n x_j^{E_{ji}}\,, 
$$ 
i.e. it corresponds to the transpose $E^T$ of the matrix $E$.
If the invertible polynomial $f$ is non-degenerate, then $\widetilde{f}$ is non-degenerate as well. 

\begin{definition} 
The (diagonal) {\em symmetry group} of the invertible polynomial $f$ is the group 
$$ 
G_f=\{(\lambda_1, \ldots, \lambda_n)\in (\CC^*)^n: 
f(\lambda_1 x_1, \ldots, \lambda_n x_n)= f(x_1, \ldots, x_n)\}\,, 
$$  
i.e. the group of diagonal linear transformations of $\CC^n$ preserving $f$. 
\end{definition} 

For an invertible polynomial $f(x_1, \ldots, x_n)=\sum\limits_{i=1}^n \prod\limits_{j=1}^n x_j^{E_{ij}}$ 
the symmetry group $G_f$ is finite. Its order $\vert G_f\vert$ is equal to $d=\det E$ (see \cite{Kreuzer, ET2}). 
The (natural) monodromy transformation of $f$ is induced by an element of the symmetry group $G_f$.

For a finite abelian group $G$, let $G^*=\mbox{Hom\,}(G,\CC^*)$ be its group of characters. 
(As abelian groups $G$ and $G^*$ are isomorphic, but not in a canonical way.) One can show (see, e.g., \cite{BH2, EG, ET2}) that the symmetry group $G_{\widetilde{f}}$ of the Berglund-H\"ubsch transpose $\widetilde{f}$ of an invertible polynomial $f$ is canonically isomorphic to $G_f^*$. 

For a subgroup $H\subset G$ its {\em dual} (with respect to $G$) $\widetilde{H}\subset G^*$ is the kernel of the natural map $i^*:G^*\to H^*$ induced by the inclusion $i:H\hookrightarrow G$ (see \cite{BH2, ET2}). One has $\vert H\vert\cdot\vert \widetilde{H}\vert = \vert G\vert$.

\section{Reduced orbifold Euler characteristic} 
For a ``relatively good'' topological space $X$, say, a union of cells in a finite $CW$-complex or a quasi-projective complex or real analytic variety, its Euler characteristic is  
$$ 
\chi(X)=\sum_i (-1)^i \dim H^i_c(X;\RR), 
$$ 
where $H^i_c(X;\RR)$ are the cohomology groups with compact support. (This Euler characteristic is additive. For quasi-projective complex analytic varieties it coincides with the one defined through the usual cohomology groups.)  

Let $G$ be a finite group acting on $X$. For a subgroup $H\subset G$, let $X^H\subset X$ be the set of fixed points of the subgroup $H$: $X^H=\{x\in X: gx=x \mbox{ for all } g\in H\}$. The {\em orbifold Euler characteristic} of the pair $(X, G)$ is defined by 
$$ 
\chi(X,G)=\frac{1}{\vert G\vert}\sum_{(g,h):gh=hg} \chi(X^{\langle g,h\rangle}) 
$$ 
where $\langle g,h\rangle$ is the subgroup of $G$ generated by $g$ and $h$ (see \cite{AS,BD,DHVW1,DHVW2,HH,Reid,Roan}). 
The orbifold Euler characteristic is an additive function on the Grothendieck ring of quasi-projective $G$-varieties.  

Now let $G$ be abelian. We will use the notion of the reduced orbifold Euler characteristic of a $G$-space $X$. The usual reduced (modulo a point) Euler characteristic is $\overline{\chi}(X)=\chi(X)-1=\chi(X)-\chi({\rm pt})$. In the $G$-equivariant setting the role of the point is played by the one point set with the trivial $G$-action~--- the unit in the Grothendieck ring of $G$-varieties. Its orbifold Euler characteristic is equal to $\vert G\vert$~--- the number of elements of $G$ (for an abelian $G$). Therefore the {\em reduced orbifold Euler characteristic} of a $G$-set $X$ will be defined as 
$$ 
\overline{\chi}(X,G)= \chi(X,G)- \vert G\vert. 
$$  

\begin{remark} 
For a germ $F: (\CC^n,0) \to (\CC,0)$ with an isolated critical point at the origin, its Milnor number is equal to $(-1)^{n-1} \overline{\chi}(V_F)$, where $V_F$ is the Milnor fibre. Therefore for a non-degenerate invertible polynomial $f$ with a group of symmetries $G$, the reduced orbifold Euler characteristic $\overline{\chi}(V_f,G)$ of the Milnor fibre $V_f$ of $f$ multiplied by $(-1)^{n-1}$ can be regarded as an orbifold Milnor number of $(f,G)$.
\end{remark} 

\section{Orbifold Euler characteristic for invertible polynomials} 
Let $f(x_1, \ldots, x_n)$ be an invertible polynomial and let $G$ be a subgroup of the group $G_f$
of symmetries of $f$. Let $(\widetilde{f}, \widetilde{G})$ be the Berglund--Henningson
dual of the pair $(f,G)$, i.e., $\widetilde{G}$ is the subgroup of $G_{\widetilde{f}}=G_f^*$
dual to $G$. Let $V_f=\{x\in\CC^n:f(x)=1\}$ be the Milnor fibre of the polynomial $f$.
This is a complex analytic manifold with a $G$-action. 

\begin{theorem} 
One has 
$$ 
\overline{\chi}(V_{\widetilde{f}}, \widetilde{G})=(-1)^n\overline{\chi}(V_f,G)\,. 
$$ 
\end{theorem} 

\begin{proof} 
For a subset $I\subset I_0=\{1,2,\ldots, n\}$, let 
$$(\CC^*)^I:= \{(x_1, \ldots, x_n)\in \CC^n: x_i\ne 0 \mbox{ for }i\in I, x_i=0 \mbox{ for }i\notin I\}$$
be the corresponding coordinate torus. 
One has $V_f=\coprod\limits_{I\subset I_0}V_f\cap(\CC^*)^I$. Each $V_f\cap(\CC^*)^I$ is invariant with respect to the $G$-action. Therefore  
$$ 
\overline{\chi}(V_f,G)=\sum_{I\subset I_0}\overline{\chi}(V_f\cap(\CC^*)^I,G)\,. 
$$ 

Let $G_f^I\subset G_f$ and $G_{\widetilde{f}}^I\subset G_{\widetilde{f}}=G_f^*$ be the isotropy subgroups of the actions of the symmetry groups $G_f$ and $G_{\widetilde{f}}$ on the torus $(\CC^*)^I$ respectively. (All points of the torus $(\CC^*)^I$ have one and the same isotropy subgroup.) 

Let $\ZZ^n$ be the lattice of monomials in the variables $x_1$, \dots, $x_n$ ($(k_1, \ldots, k_n)\in \ZZ^n$ corresponds to the monomial $x_1^{k_1}\cdots x_n^{k_n}$) 
and let $\ZZ^I:=\{(k_1, \ldots, k_n)\in \ZZ^n: k_i=0 \mbox{ for }i\notin I\}$. 
For a polynomial $F$ in the variables $x_1$, \dots, $x_n$, let $\mbox{supp\,} F\subset \ZZ^n$ be the set of monomials (with non-zero coefficients) in $F$. 

One has
\begin{eqnarray}
\overline{\chi}(V_f,G) & =& \frac{1}{\vert G\vert}
\sum_{I\subset I_0} \chi(V_f\cap(\CC^*)^I)\cdot \vert G_f^I\cap G\vert^2-\vert G\vert \nonumber \\
& = & \sum_{I\subset I_0}\frac{1}{\vert G\vert}\cdot \chi((V_f\cap(\CC^*)^I)/G_f)\cdot
\frac{\vert G_f\vert}{\vert G_f^I\vert}\cdot\vert G_f^I\cap G\vert^2-\vert G\vert. \label{orbi}
\end{eqnarray}

The coefficient $\chi((V_f\cap(\CC^*)^I)/G_f)$ is different from zero if and only if $\chi(V_f\cap(\CC^*)^I) \neq 0$. This is the case if and only if 
$\mbox{supp\,} f\cap \ZZ^I$ consists of $\vert I\vert$ points,
i.e. if in $f$ there are exactly $\vert I\vert$ monomials in variables $x_i$ with $i\in I$. 

Let $\vert\mbox{supp\,} f \cap \ZZ^I\vert=\vert I\vert =:k$, and let $0<k<n$, i.e. $I$ is a proper subset of $I_0$. In \cite{EG} it was shown that in this case
$\chi((V_f\cap(\CC^*)^I)/G_f)=(-1)^{\vert I\vert-1}$. Renumbering the coordinates $x_i$ and 
the monomials in $f$ permits to assume that 
$$
E = \left( \begin{array}{cc} E_I & 0 \\ \ast & E_{\overline{I}} \end{array} \right),
$$ 
where $E_I$ and $E_{\overline{I}}$ are square matrices of sizes $k \times k$ and $(n-k) \times (n-k)$ respectively. In this case
$\vert\mbox{supp\,} \widetilde{f} \cap \ZZ^{\overline{I}}\vert=n-k=\vert {\overline{I}}\vert$.
Moreover, according to \cite[Lemma 1]{EG}, one has
$G_{\widetilde{f}}^{\overline{I}} = \widetilde{G_f^I}\subset G_{\widetilde{f}}$ and therefore $|G_f^I| \cdot |G_{\widetilde{f}}^{\overline{I}}| = |G_f|$.
The summand on the right hand side of the equation~(\ref{orbi}) corresponding to $I$ is equal to
$$
\frac{(-1)^{k-1}}{\vert G\vert}\cdot\frac{\vert G_f\vert}{\vert G_f^I\vert}\cdot\vert G_f^I\cap G\vert^2.
$$ 
The subgroup $G_{\widetilde{f}}^{\overline{I}}\cap \widetilde{G}$ of $G_{\widetilde{f}}$ is dual to the subgroup $G_{f}^{I}+G$ of $G_{f}$. Therefore
$$
\vert G_{\widetilde{f}}^{\overline{I}}\cap \widetilde{G}\vert=\frac{\vert G_f\vert \vert G_f^I\cap G\vert}{\vert G_f^I\vert\vert G\vert}.
$$
Thus the summand on the right hand side of the analogue of the equation~(\ref{orbi}) for $\overline{\chi}(V_{\widetilde{f}}, \widetilde{G})$ corresponding to $\overline{I}$ is equal to
$$
\frac{(-1)^{n-k-1}}{\vert \widetilde{G}\vert}\cdot\frac{\vert G_{\widetilde{f}}\vert}{\vert G_{\widetilde{f}}^{\overline{I}}\vert}\cdot\vert G_{\widetilde{f}}^{\overline{I}}\cap \widetilde{G}\vert^2=
(-1)^{n-k-1}\frac{\vert G_f\vert\cdot\vert G_f^I\cap G\vert^2}{\vert G_f^I\vert\cdot\vert G\vert}.
$$
One can see that it differs from the corresponding summand in $\overline{\chi}(V_f,G)$
only by the sign $(-1)^n$. The isotropy subgroup $G_f^{I_0}$ is trivial and therefore the summand in $\overline{\chi}(V_f,G)$ corresponding to $I=I_0$ is equal to 
$$
\frac{(-1)^{n-1}}{\vert G\vert}\cdot \vert G_f\vert=(-1)^{n-1}\vert\widetilde{G}\vert,
$$
i.e. to the 
multiplied by $(-1)^n$
term in $\overline{\chi}(V_{\widetilde{f}}, \widetilde{G})$ outside of the sum.
Analogously,  the summand in $\overline{\chi}(V_{\widetilde{f}}, \widetilde{G})$ corresponding to $\overline{I}=I_0$ is equal to the term in $\overline{\chi}(V_f,G)$ outside of the sum multiplied by the same factor. This proves the statement.
\end{proof}


\bigskip
\noindent Leibniz Universit\"{a}t Hannover, Institut f\"{u}r Algebraische Geometrie,\\
Postfach 6009, D-30060 Hannover, Germany \\
E-mail: ebeling@math.uni-hannover.de\\

\medskip
\noindent Moscow State University, Faculty of Mechanics and Mathematics,\\
Moscow, GSP-1, 119991, Russia\\
E-mail: sabir@mccme.ru

\end{document}